\documentclass[a4paper]{article}

\usepackage[utf8]{inputenc}
\usepackage[T1]{fontenc}
\usepackage{textcomp}
\usepackage{amsmath, amssymb}
\usepackage{amsthm}
\usepackage{mathrsfs}
\usepackage{esint}

\usepackage{import}
\usepackage{xifthen}
\usepackage{pdfpages}
\usepackage{transparent}
\usepackage{hyperref}
\usepackage{subcaption}
\usepackage{caption}

\newtheorem{problem}{Problem}[section]
\newtheorem{question}{Question}[section]
\newtheorem{theorem}[problem]{Theorem}
\newtheorem{proposition}[problem]{Proposition}

\newtheorem{remark}[problem]{Remark}

\theoremstyle{definition} 
\newtheorem{definition}[problem]{Definition}
\definecolor{myblack}{RGB}{40,44,52}
\definecolor{mywhite}{RGB}{171,178,191}
\begin{document}
\title{Uniformization of surfaces with boundary and the application to the triple junction surfaces with negative Euler characteristic}
\author{Gaoming Wang}
\date{}
\maketitle

\begin{abstract}
	The conformal structure on minimal surfaces plays a key role in studying the properties of minimal surfaces. Here we extend the results of uniformization of surfaces with boundary to get the (weak) uniformization results for triple junction surfaces.
\end{abstract}

\section{Introduction}%
\label{sec:introduction}

The classical uniformization theorem says that for every conformal structure on Riemann surface $\Sigma$ without boundary, we can find a metric of constant Gaussian curvature $-1,0$ or $1$ depending on the genus of $\Sigma$. There is a lot of work related to the uniformization of Riemann surfaces, including the prescribed Gaussian curvature problem \cite{berger1971riemannian,kazdan1974curvature, borer2015large, chen2003gaussian} and the complete surface case \cite{mazzeo2002curvature}. 

After the work of \cite{wang2021curvature}, we wonder could we do the uniformization for triple junction surfaces? 
It turns out we do have some partial results regarding the uniformization of triple junction surfaces. 

Suppose we have a triple junction surface $M=(\Sigma_1,\Sigma_2,\Sigma_3;\Gamma)$ with a compatible metric $g=(g_1,g_2,g_3)$ on it (see Section \ref{sec:preliminary} for precise definitions of these notations), each $\Sigma_i$ is orientable, we can prove the following (weak) uniformization for $M$. 
\begin{theorem}
	
	If $\chi(M)\le 0$ and $\Gamma$ has only one component, then we can find a new metric $\overline{g}=(\overline{g}_1, \overline{g}_2,\overline{g}_3)$ that each $\overline{g}_i$ is a hyperbolic metric on $\Sigma_i$ such that the following holds,
	\begin{itemize}
		\item $\overline{g}_i$ has constant geodesic curvature on boundary $\kappa_i$.
		\item The boundary length $L_i$ of $\partial \Sigma_i$ under metric $\overline{g}_i$ satisfies $L_i=L_j$ for $i\neq j$ (same boundary length).
		\item The sum of geodesic curvature on $\Gamma$ is 0. That is $\sum_{i =1}^{3}\kappa_i=0$.
	\end{itemize}
	\label{thm:uniform_triple}
\end{theorem}

The precise statement of this theorem is given by Theorem \ref{thm_weak_uniformization_for_chi_m_le_0_}. 

The reason we care about the uniformization of triple junction surfaces is, we want to learn more about the conformal structure of triple junction surfaces, especially about the minimal triple junction surfaces.
In the classical theory of minimal surfaces, the conformal structure on minimal surfaces always plays an important role in the study of minimal surfaces. 
For example, noting that the Gauss map of a minimal surface immersed in $\mathbb{R}^3 $ is indeed a conformal map, we can determine the conformal structure for a complete minimal surface in $\mathbb{R}^3 $ with finite total curvature, see for example \cite{OssermanRobert1986Asom}. We note that the uniqueness of topological minimal sphere in $\mathbb{S}^3$ relies on the conformal structure on $\mathbb{S}^2$. 
Here is a survey \cite{meeks2004conformal} of some conformal properties about minimal surfaces.
So it is worth learning about the conformal structures on triple junction surfaces.

Besides, the extremal metric problem of the first eigenvalue is still related to the conformal structure. In \cite{hersch1970quatre} for $\mathbb{S}^2$ case and in \cite{li1982new} for $\mathbb{RP}^2$ case, they showed the extremal metrics on those surfaces are the constant curvature metrics. In particular, the eigenfunctions corresponding to the first eigenvalues of extremal metrics are usually related to the immersion of a minimal surface into the sphere, see for example \cite{li1982new, yang1980eigenvalues} for details. 
So the study of triple junction surfaces in the sphere is still required us to learn more about the conformal structures and extremal metrics on it. Uniformization is the first step to do that.

Based on the definition of weak uniformization (see Section \ref{sec:preliminary}), we can see the key in the uniformization for triple junction surfaces is to do the uniformization of surface with boundary at first. Actually, there are also some results regarding the uniformization of surface with boundary, see for examples \cite{osgood1988extremals, osgood1988compact,osgood1989moduli,khuri1991heights,kim2008surfaces} including the study of moduli space of surfaces with boundary.

In the work of Osgood, Philips and, Sarnak \cite{osgood1988extremals}, they've considered two ways (fixing the area or fixing the boundary length) of finding extremal metrics to get two kinds of uniformization, one is the hyperbolic metric with geodesic boundary curves and another one is the flat metric with constant geodesic curvature on the boundary. Moreover, Brendle \cite{brendle02curvatureflow,brendle02afamilyofcurvatureflow} has considered a family of curvature flow of metrics and shown that the metric will converge to a metric with constant interior Gaussian curvature $k$ and constant geodesic curvature $c$ on the boundary. Note that $k$ and $c$ have the same sign in the work of Brendle. 

The remarkable results for the uniformization of surfaces with boundary were due to Rupflin \cite{rupflin2021hyperbolic}.
In her work, she extended the uniformization to the case $c$ and $k$ could have different signs. Our second main theorem is like an extension of Rupflin's work of surface with only one boundary component. More precisely, we will prove the following theorem (see Theorem \ref{thm:uniform_surface} for more precise statement),
\begin{theorem}
	
	Let $(\Sigma,g)$ be a smooth compact oriented surface with boundary $\partial \Sigma$ and negative Euler characteristic. 
	Suppose $\partial \Sigma$ has only one component.
	Then for any $k\le 0, c<\sqrt{-k}$, there is a unique metric $g_{k,c}$ conformal to $g$ such that it has constant Gaussian curvature $k$ and constant geodesic curvature $c$ on $\partial \Sigma$. 
	
	Moreover, if we let $L(k,c)$ be the length of boundary under metric $g_{k,c}$, then $L$ is a continuous function defined on domain $\mathcal{D}= \{ (k,c)\in \mathbb{R}^2 : k\le 0, c<\sqrt{-k}\}$. In particular, we have the following monotonicity properties for $L$ when fixing $k$,
	\begin{itemize}
		\item $L(k, \cdot)$ is strictly increasing with $L(k,-\infty)
			=0,L(k,\sqrt{-k})=+\infty$.
		\item $\hat{L}(k,c):= c L(k,c)$ is strictly increasing with respect to the variable $c$ with $\hat{L}(k,\sqrt{-k})=
			+\infty$ for $k<0$ and $\hat{L}(k,-\infty)=2\pi \chi(\Sigma)$ for $k\le 0$.
	\end{itemize}
	\label{thm:surface_with_boundary_introduction}
\end{theorem}

Here we use notation $L(k,c):= \lim_{(k',c')\rightarrow (k,c)} 
L(k',c')$ if $(k,c)$ is not in $\mathcal{D}$. 

Actually, existence and uniqueness are already known to us. The main results developed in this paper are to establish the relation of geometric properties of $g_{k,c}$ and the choice of $k,c$. In particular, we are interested in the properties of function $L(k,c)$. By the work of Rupflin \cite{rupflin2021hyperbolic}, we already know the results of case $k\le 0,0\le c<\sqrt{-k}$. So in this paper, we will focus on the case $k\le 0$ and $c\le 0$. 
Up to a scaling factor, we only need to focus on the cases of fixing $k=-1$ and let $c$ change, and fixing $c=-1$ and let $k$ change. The key result in the application to the weak uniformization of triple junction surfaces is the monotonicity property and the limit behavior of function $\hat{L}(-1,c)$. Especially the limit $\hat{L}(-1,-\infty)=2\pi\chi(\Sigma)$ can only be computed after studying the case fixing $c=-1$.
Unlike the complicated trace theorem and energy estimate considered in the work \cite{rupflin2021hyperbolic}, we only need the standard Maximum Principle to develop the energy estimate for the case of $k\le 0,c\le 0$ to prove Theorem \ref{thm:surface_with_boundary_introduction}. 

For the case of positive Euler characteristic, including $\chi(\Sigma)>0$, the surface with boundary, and $\chi(M)>0$, the triple junction surfaces, things get complicated. Usually, we will try to minimize a suitable energy functional to get the existence of hyperbolic metric when Euler characteristic is negative, but this method is not work in this case, even in the usual surface without boundary case $\mathbb{S}^2$. Readers may refer to \cite{chang1988conformal,ding2001proof,mazzeo2002curvature,struwe2005flow} for the uniformization on sphere, including the prescribed Gaussian curvature problem (known as Nirenberg's problem).
In particular, Malchiodi \cite{malchiodi2017variational} 
considered the singular version of uniformization with possible positive Gaussian curvature after uniformization.

In Section \ref{sec:preliminary}, we will give the definitions of triple junction surfaces and two types of uniformization for triple junction surfaces. In Section \ref{sec:uniformization_of_surface_with_boundary}, we will consider the uniformization for surfaces with boundary to prove Theorem \ref{thm:surface_with_boundary_introduction}.
Then in Section \ref{sec:weak_uniformization_of_triple_junction_surface}, we will use the uniformization results for surfaces with boundary to finish the proof of weak uniformization of triple junction surfaces.

\section{Preliminary}%
\label{sec:preliminary}

In this section, we will give the definition of triple junction surfaces and the basic concepts of uniformization of triple junction surfaces.

We will follow the definitions of triple junction surfaces from \cite{wang2021curvature}. Instead of multiple junction surfaces, we only focus on the triple junction surfaces in this paper.

Let $\Sigma_1,\Sigma_2,\Sigma_3$ to be three (orientable)
2-dimensional surfaces with smooth boundary $\partial \Sigma_1,\partial \Sigma_2,\partial \Sigma_3$.
Let $\Gamma$ be a one-dimensional curves.

\begin{definition}
	We say $M=(\Sigma_1,\Sigma_2,\Sigma_3;\Gamma)$ is a \textit{triple junction surface} if in addition, there are three diffeomorphism $\varphi_i: \Gamma\rightarrow \partial \Sigma_i$ with $i=1,2,3$. This means $\partial \Sigma_1,\partial \Sigma_2,\partial \Sigma_3$ are diffeomorphic to each other.
	\label{def:triple}
\end{definition}

Now we can define the metric on $M$.

\begin{definition}
	We say a tuple $g=(g_1,g_2,g_3)$ where $g_i$ is a metric on $\Sigma_i$ defines a metric on $M$ if there is a metric $g_\Gamma$ on $\Gamma$ such that $\varphi_i: \Gamma\rightarrow \partial \Sigma_i$ becomes an isometric for each $i=1,2,3$.

\end{definition}

In the later on, we can identify the $\partial \Sigma_i$ with $\Gamma$ if there is no confusion. For example, when we talk about the restriction of metric $g_i$ on $\Gamma$, we actually means the metric $g_i|_{\partial \Sigma_i}$ on $\partial \Sigma_i$, or the pull-back metric $\varphi_{i}^*g_i|_{\partial \Sigma_i}$.

Given metric $g$ on $M$, we can define the unit outer normal vector field $\tau_i$ of $\Gamma$ in $\Sigma_i$.
Usually, the $\tau_i$ is chosen such that the following holds.
\begin{itemize}
	\item $g_i(\tau_i,\tau_i)=1$ along $\Gamma$.
	\item $g_i(\tau_i,\alpha)=0$ for any $\alpha $ in the tangent space of $\Gamma$.
	\item $\tau_i$ points outward of $\Sigma_i$.
\end{itemize}

Besides, we can choose a unit tangent vector field $\eta$ along $\Gamma$ since $\Gamma$ is a curve. Note that since $g_i$ is the same when restricting on $\Gamma$, we do not need to make the difference of tangent vector field of $\Gamma$ in different surface $\Sigma_i$.

Now, for any $p \in \Gamma$, we can define the geodesic curvature for $\Gamma$ in $\Sigma_i$ by
\[
	\kappa_i(p)=g_i(\nabla_{\eta}^{\Sigma_i}\eta, \tau_i)
\]
where $\nabla_{X}^{\Sigma_i}Y$ be the coderivative on $\Sigma_i$ with respect to metric $g_i$.

\begin{definition}
	We say the metric $g$ on $M$ is \textit{compatible} with $M$ if the sum of geodesic curvature of $\Gamma$ is zero. That is saying
\[
		\sum_{i =1}^{3}\kappa_i(p)=0\quad \text{ for all } p \in \Gamma.
	\]
	\label{def_metric}
\end{definition}

Actually, the compatible condition for metric is coming from the fact that whether $M$ can be minimality immersed to a big ambient manifold as a triple junction surface. For example, if $M$ is indeed a minimal immersed triple junction surface in $\mathbb{R}^3 $, then the sum of outer normal $\tau_i$ should be zero. This will imply the sum of geodesic curvature $\kappa_i$ is zero after a simple calculation.
So this is a necessary condition if we want to study the metric and conformal structure of a minimal triple junction surface $M$ intrinsically.

\begin{remark}
	The compatible condition says we have a restriction of derivative of $g_i$ along the direction $\tau_i$. Roughly speaking, this compatible condition says we have the sum of $\frac{\partial g_i}{\partial \tau_i}$ is zero along $\Gamma$.
\end{remark}

Now we can talk about the uniformization for a triple junction surface $M$. There are two kinds of uniformization we can talk about. 
We list them as follows. 
\begin{itemize}
	\item (Strong uniformization) Find a new metric $\overline{g}=(\overline{g}_1,\overline{g}_2, \overline{g}_3)$ such that the following holds.
		\begin{itemize}
			\item $\overline{g}$ compatible with $M$.
			\item $\overline{g}_i$ is conformal equivalent to $g_i$ on $\Sigma_i$.
			\item Gaussian curvature of $g_i$ is a fixed constant $k \in \mathbb{R} $ where $k$ does not rely on $i$.
		\end{itemize}
		
		\item (Weak uniformization) Find a new metric $\overline{g}=(\overline{g}_1,\overline{g}_2, \overline{g}_3)$ and new diffeomorphisms on boundary $\overline{\varphi}_i: \Gamma\rightarrow \partial \Sigma_i$ such that the following holds
		\begin{itemize}
			\item $\overline{g}$ is compatible with $M$ under new boundary diffeomorphisms.
			\item $\overline{g}_i$ is conformal equivalent to $g_i$ on $\Sigma_i$.
			\item Gaussian curvature of $g_i$ is a fixed constant $k \in \mathbb{R} $ where $k$ does not rely on $i$.
			\item Geodesic curvature $\kappa_i$ is constant along $\Gamma$. Different $i$ might have different constant here.
		\end{itemize}
\end{itemize}

Roughly speaking, in the strong uniformization, we will only uniform interior Gaussian curvature but we will fix the boundary diffeomorphisms. This is a reasonable way to do uniformization since in general we still want to keep the local structure after uniformization in some sense. But before we start the strong uniformization, we might need to develop the full elliptic PDE theories on triple junction surfaces, especially the regularity of solutions. Besides that, one drawback of strong uniformization is, since the different diffeomorphisms on the boundary will give the different conformal structure(in the sense of an equivalent class of metric). But the space of diffeomorphisms is quite large in general, at least it is infinite-dimensional. So the Moduli space related to strong uniformization is quite complicated.

But for weak uniformization, we relax the boundary condition for $M$. Roughly speaking, we only require the new metric $\overline{g}$ to have the same boundary length when restricting on $\Gamma$. So in this weak condition, we can uniform the interior Gaussian curvature and the boundary geodesic curvature at the same time. The goal of weak uniformization is just to find this kind of $\overline{g}$ such that the sum of the boundary geodesic curvatures is zero.

So we can see that weak uniformization is more like doing uniformization on each $\Sigma_i$ and to see if they can be matched. We do not need to do PDE stuff on triple junction surfaces anymore. From the Teichm\"uller theory for surfaces with boundary, we can see if we want to study the Teichm\"uller space of triple junction surface, it should be finite-dimensional at least, which is easier than the case of strong uniformization.

In this paper, we will only focus on weak uniformization.

Now we can consider the relation of Euler characteristic and the Gaussian curvature of triple junction surfaces.
Recall that the Euler characteristic is well-defined for each triple junction surface $M$ since we can view it as a two-dimensional finite CW-complex by identifying boundaries of $\Sigma_i$. By the results in algebraic topology, we can see
\[
	\chi(M)=\sum_{i =1}^{3}\chi(\Sigma_i).
\]

Now from the Gauss-Bonnet formula, we have
\begin{align*}
	\chi(M)={} & \sum_{i=1 }^{3}\chi(\Sigma_i)
	=\sum_{i =1}^{3}\left( \int_{ \Sigma_i} K(g_i)dA_i+ \int_{ \Gamma} \kappa_ids \right) \\
	={} & \sum_{i =1}^{3}\int_{ \Sigma_i} K(g_i)dA_i 
\end{align*}
where $K(g_i)$ is the Gaussian curvature of $g_i$, $dA_i(ds)$ are the area elements (the length element) of $\Sigma_i$ ($\Gamma$). Here we've used $g$ is compatible with $M$. So basically, after uniformization, the interior Gaussian curvature is completely determined by $\chi(M)$ upto a positive rescaling factor. 

In the main Theorem \ref{thm:uniform_triple}, although we will only consider the case $\chi(M)\le 0$, it still contains most of the triple junction surfaces with connected triple junction in some sense.
Indeed, for $\Sigma_i$, we have
\begin{align}
	\chi(\Sigma_i)=2-2\text{genus}(\Sigma_i)
	-\sharp(\text{components of }
	\partial \Sigma_i).
	\label{eq:Euler_formula}
\end{align}

So since we require $\Gamma$ to have only one component, $\chi(\Sigma_i)$ is at most 1. There are only two types of triple junction surfaces with $\chi(M)>0$. The first one is each of $\Sigma_i$ is a disk, and the second one is one of $\Sigma_i$ is a torus removed a disk and the remaining two are two disks. 
%On the other
%hand, if $\Gamma$ has more than one
%component, then $\chi(\Sigma_i)\le 0$
%all the time. This means 
%$\chi(M)\le 0$. So the theorem can
%be applied in this case.

Moreover, for the case of $M$ composed by three disks, the uniformization result is known to us. Actually, we have the following result. 
\begin{proposition}
	If $M=(\Sigma_1,\Sigma_2,\Sigma_3;
	\Gamma)$ is a triple junction surface such that each $\Sigma_i$ is homeomorphic to a disk, then we can uniform it into a triple junction surface with Gaussian curvature 1 weakly.
	\label{prop:uniform_3disk}
\end{proposition}
\begin{proof}
	Since the disk has a unique conformal structure, we can just choose the standard half-sphere metric $\overline{g}_i$ for each $i$ and they are conformal equivalent to the metric $g_i$.
	Now we can construct the isometries between $\partial \Sigma_i$ and $\mathbb{S}^1$, the unit circle.
	
	This metric is compatible with $M$ since the geodesic curvature for each boundary is zero. So we can uniform $M$ to be a triple junction surface with Gaussian curvature 1 interior.
\end{proof}

In summary, we have only one topological type that we do not know if we can do the uniformization. 

For the proof of uniformization for triple junction surface with $\chi(M)\le 0$, we need to consider the uniformization of surfaces with boundary first.

\section{Uniformization of surfaces with boundary}%
\label{sec:uniformization_of_surface_with_boundary}

In this section, we will consider the surface $\Sigma$ with smooth boundary $\partial \Sigma$ which has only one component.
Without loss of generality, we assume $\Sigma$ is not a disk since the conformal structure on the disk is completely known to us. Note that the Euler characteristic of $\Sigma$ is always an odd number if $\Sigma$ has only one boundary component by formula (\ref{eq:Euler_formula}), we only need to consider the case $\chi(\Sigma)<0$.

So given such $\Sigma$ with metric $g$ on it, our fundamental question is the following,
\begin{question}
	Given which kind of $k \in \mathbb{R} , c \in \mathbb{R} $, we can find a conformal metric $\overline{g}$ such that the Gaussian curvature of $\Sigma$ with metric $\overline{g}$ is just $k$ and the geodesic curvature of boundary is $c$ under metric $\overline{g}$? If we can find such of $\overline{g}$, how about the uniqueness, and moreover, how does this metric $\overline{g}$ depend on $k, c$?
\end{question}

If we write the conformal metric $\overline{g}=e^{2u}g$, let $K(x):=K_g(x)$ be the Gaussian curvature of metric $g$, $\kappa_g(x)$ be the geodesic curvature of $\partial \Sigma$ under the metric $g$. Recall that the Gaussian curvature and the geodesic boundary curvature are characterised by
\begin{align}
	K_{\overline{g}} ={} & 
	e^{-2u}\left( K_{g}-\Delta u\right) \\
	\kappa_{\overline{g}} ={} & 
	e^{-u}\left( \kappa_{g}+ \frac{\partial u}{\partial \tau}\right) 
\end{align}
under conformal change of metric. Here, we still use $\tau:=\tau_g$ to denote the unit outer normal of $\partial \Sigma$ in metric $g$ and the Laplacian $\Delta:=\Delta_g$ being the Laplace-Beltrami operator with respect to the metric $g$. 

Using these two formulas, we know the existence and uniqueness of $\overline{g}$ are equivalent to the existence and uniqueness of the following PDE problem,
\begin{align}
	\begin{cases}
		-\Delta u=ke^{2u}-K, & \text{ in }\Sigma, \\
	\frac{\partial u}{\partial \tau}=c
		e^{u}-\kappa, & \text{ on }\partial \Sigma.
	\end{cases}
	\label{eq:pde_unform}
\end{align}

Recall that the Gauss-Bonnet formula says
\[
	2\pi\chi(\Sigma)=\int_{ \Sigma} KdA+\int_{ \partial \Sigma} \kappa ds.
\]

So if the problem (\ref{eq:pde_unform}) does have a solution, then $k,c$ should satisfies
\begin{align}
	2\pi\chi(\Sigma)=k \mathrm{Area}_{\overline{g}}
	(\Sigma)+c
	\text{Length}_{\overline{g}}(\partial \Sigma)
	\label{eq:Gauss_Bonnet}
\end{align}

Since we do not consider the disk case, we have $2\pi\chi(\Sigma)\le 0$. One of the necessary condition for (\ref{eq:pde_unform}) having a solution is $k,c$ cannot greater than 0 at the same time.

So study of problem (\ref{eq:pde_unform}) can be divided into three cases.
\begin{itemize}
	\item $k\le 0, c\le 0$. This case is easy to deal with since they have the same sign in the formula (\ref{eq:Gauss_Bonnet}). Most of the previous work belongs to this case. We will still focus on this case and try to give some more precise results, especially about how the geometric properties of new metrics related to the choice of different $k,c$.
	\item $k<0, c>0$. This case is much harder than the first case. The remarkable contribution to this case was made by M. Rupflin \cite{rupflin2021hyperbolic}. 
		To be more precise, when $k=-1$ and $c \in [0,1)$, she not only showed the existence and uniqueness of problem (\ref{eq:pde_unform}), but also established the relation between the boundary length and the prescribed constant $c$.
	\item $k>0, c<0$. Nothing is known to us in this case.
\end{itemize}

In this section, we will give the uniformization for the first two cases. In particular, we will prove the following theorem.

\begin{theorem}
	Let $(\Sigma,g)$ be a smooth compact oriented surface with boundary $\partial \Sigma$ and negative Euler characteristic. We let $\mathcal{D}:=\{ (k,c) \in \mathbb{R}^2 : k\le 0, c< \sqrt{-k}\}$ be a domain in $\mathbb{R}^2$. Then for any $(k,c)\in \mathcal{D}$, there exists a unique metric $g_{k,c}$ conformal to $g$ such that
	\begin{itemize}
		\item Gaussian curvature of $g_{k,c}$ will be identically to $k$.
		\item Geodesic curvature of boundary under metric $g_{k,c}$ is identically to $c$ along $\partial \Sigma$.
	\end{itemize}
	
	Moreover, we can establish the relation of the area and the boundary length of $\Sigma$ under metric $g_{k,c}$ when $k,c$ change in the following ways.

	We use the following notations,
	\begin{align*}
		L(k,c):={} & \mathrm{Length}
		_{g_{k,c}}
		(\partial \Sigma)\\
		A(k,c):={} & \mathrm{Area}
		_{g_{k,c}}(\Sigma)
	\end{align*}
	
	Then $L(k,c), A(k,c)$ are all continuous functions on the domain $\mathcal{D}$. 
	In particular, we have the monotonicity properties for the following four functions.
	\begin{itemize}
		\item For function $L:\mathcal{D}\rightarrow \mathbb{R} $, it is strictly increasing along both positive directions of $k,c$ with the following asymptotic behavior,
			\begin{align}
\arraycolsep=1.4pt
				\begin{array}{rll}
					L(-\lambda^2,\lambda)=&
				+\infty & \text{ for any }
				\lambda>0,\\
				L(k,-\infty)= & 0
				&\text{ for any }k\le 0,\\
				L(0,c)=&\frac{2\pi \chi(\Sigma)}{c}\quad  & 
			\text{ for any }c<0,\\
			L(-\infty,c)=& 0 & 
			\text{ for any }c \in \mathbb{R}.
			\end{array}
			\label{eq:L_asy}
		\end{align}
	\item For function $A:\mathcal{D}\rightarrow \mathbb{R} $, it is strictly increasing along both positive directions of $k,c$ with the following asymptotic behavior,
		\begin{align}
			\arraycolsep=1.4pt
			\begin{array}{rll}
					A(-\lambda^2,\lambda)=&
				+\infty & \text{ for any }
				\lambda> 0,\\
				A(k,-\infty)= & 0
				&\text{ for any }k\le 0,\\
				A(0,c)=&\frac{A_0
				}{c^2}\quad  & 
			\text{ for any }c<0,\\
			A(-\infty,c)=& 0 & 
			\text{ for any }c \in \mathbb{R}.
			\end{array}
			\label{eq:A_asy}
		\end{align}
		where $A_0= A(0,-1)$.
	\item For function $ \hat{L}(k,c):=cL(k,c): \mathcal{D}\rightarrow \mathbb{R} $, it is strictly increasing along the positive direction of $c$ when $k<0$ with the following asymptotic behavior,
		\begin{align}
			\arraycolsep=1.4pt
				\begin{array}{rll}
				\hat{L}(-\lambda^2,\lambda)=&
				+\infty & \text{ for any }
				\lambda> 0,\\
				\hat{L}(k,-\infty)= & 
				2\pi \chi(\Sigma)
				&\text{ for any }k\le 0,\\
				\hat{L}(0,c)=&2\pi \chi(\Sigma)\quad  & 
			\text{ for any }c<0,\\
			\hat{L}(-\infty,c)=& 0 & 
			\text{ for any }c \in \mathbb{R}.
			\end{array}
			\label{eq:cL_asy}
		\end{align}
		$\hat{L}$ still has the monotonicity property when fixing $c$. To be precise, $\hat{L}(k,c)$ is strictly decreasing when $c<0$ and $\hat{L}(k,c)$ is strictly increasing when $c>0$ and $\hat{L}(k,0)=0$ all the time.
			
	\item For function $ \hat{A}(k,c):=kA(k,c): \mathcal{D}\rightarrow \mathbb{R} $, it is strictly increasing along the negative direction of $c$ when $k<0$ with the following asymptotic behavior,
		\begin{align}
			\arraycolsep=1.4pt
				\begin{array}{rll}
				\hat{A}(-\lambda^2,\lambda)=&
				-\infty & \text{ for any }
				\lambda> 0,\\
				\hat{A}(k,-\infty)= & 
				0
				&\text{ for any }k\le 0,\\
				\hat{A}(0,c)=&0
				\quad  & 
			\text{ for any }c<0,\\
			\hat{A}(-\infty,c)=& 
			2\pi \chi(\Sigma)& 
			\text{ for any }c \in \mathbb{R}.
			\end{array}
			\label{eq:KA_asy}
		\end{align}
		$\hat{A}$ still has the monotonicity property when fixing $c$. To be precise, $\hat{A}(k,c)$ is strictly increasing when $c<0$ and $\hat{A}(k,c)$ is strictly decreasing when $c>0$ and $\hat{A}(k,0)=2\pi \chi(\Sigma)$ all the time.
	\end{itemize}
	\label{thm:uniform_surface}
\end{theorem}

\begin{remark}
	The above theorem still holds when $\partial \Sigma$ has more than one component. The length $L(k,c)$ will be the sum of lengths of all possible components of $\partial \Sigma$. But such a result cannot help us to solve weak uniformization for triple junction surfaces with $\Gamma$ having more than one component. See Remark \ref{rmk:multiple_boundary} for details.
\end{remark}

Note by Gauss-Bonnet formula, we have $\hat{L}+\hat{A}=2\pi \chi(\Sigma)$.
So the third and fourth cases in Theorem \ref{thm:uniform_surface} are equivalent to each other, and we only need to prove one of them.

Moreover, note that for any metric $g$ and $\lambda>0$, the scaled metric $\lambda^2 g$ has Gaussian curvature $K_{\lambda^2g}=\frac{K_g}{\lambda^2}$ and geodesic boundary curvature $\kappa_{\lambda^2g}=\frac{\kappa_g}{\lambda}$.
So we only need to prove the Theorem \ref{thm:uniform_surface} in some special case like we fix $k=-1$ or $c=-1$.
The following two theorems are what we want to prove in this section.

\begin{theorem}
	
	Let $(\Sigma,g)$ be a smooth compact oriented surface with boundary $\partial \Sigma$ and negative Euler characteristic.
	Then for any $c<1$, there is a unique metric $g_{-1,c}$ conformal to $g$ such that $(\Sigma,g_{-1,c})$ has constant Gaussian curvature $-1$ and constant geodesic curvature $c$ on the boundary.

	Moreover, the three functions $L(-1,c),A(-1,c),\hat{L}(-1,c)$ are all strictly increasing continuous function with limit $+\infty$ when $c\rightarrow 1$.
	\label{thm:fix_K}
\end{theorem}
\begin{theorem}
	Let $(\Sigma,g)$ be a smooth compact oriented surface with boundary $\partial \Sigma$ and negative Euler characteristic.
	Then for any $k\le 0$, there is a unique metric $g_{k,-1}$ conformal to $g$ such that $(\Sigma,g_{k,-1})$ has constant Gaussian curvature $k$ and constant geodesic curvature $-1$ on the boundary.

	Moreover, the three functions $L(k,-1), A(k,-1),\hat{A}(k,-1)$ are all strictly increasing continuous function with some special values $L(0,-1)=-2\pi \chi(\Sigma), \hat{A}(0,-1)=0$.
	\label{thm:fix_c}
\end{theorem}

Since $\chi(\Sigma)<0$, we can choose a hyperbolic metric $g_0$ conformal to $g$ (which is unique) with geodesic boundary curves (using the standard doubling method). In the following proof, we will always assume the base metric $g$ is a hyperbolic metric with geodesic boundary curves since our theorem only relies on the conformal structure of $\Sigma$.

So the corresponding PDE problems for the above theorems are the following
\[
	\begin{cases}
	-\Delta u =1+ke^{2u}, & 
	\text{ in }\Sigma\\
	\frac{\partial u}{\partial \tau}= ce^{u}, & \text{ on }\partial \Sigma.
	\end{cases}
	\label{eq:K-1_PDE}
\]

%\subsection{Proof for the case $k=-1$}%
%\label{sub:proof_of_the_case_k_0_1_}

At first, let's establish the existence, uniqueness from the point of view of PDEs.

\begin{proposition}
	Let $(\Sigma,g)$ be a smooth compact oriented surface with boundary $\partial \Sigma$ and $\chi(\Sigma)<0$. Suppose $g$ is a hyperbolic metric on $\Sigma$ with geodesic boundary curves. Then for any $c \in (-\infty,1)$, there is a unique weak solution $u_{c} \in H^{1}(\Sigma,g)$ of the following problem
	\begin{align}
		\begin{cases}
		-\Delta u=1-e^{2u}, & 
		\text{ in }\Sigma,\\
		\frac{\partial u}{\partial \tau}= ce^{u}, & \text{ on }\partial\Sigma.
		\end{cases}
		\label{eq:K_fixed}
	\end{align}

	Similarly, for any $k
	\in (-\infty,0]$, there is a unique weak solution $v_{k}
	\in H^{1}(\Sigma,g)$ of the following problem,
	\begin{align}
		\begin{cases}
		-\Delta u=1+ke^{2u}, & 
		\text{ in }\Sigma,\\
		\frac{\partial u}{\partial \tau}=
		-e^{u}, & \text{ on }\partial \Sigma.
		\end{cases}
		\label{eq:c_fixed}
	\end{align}
	
	Both of these two solutions are smooth upto the boundary. 
	\label{prop:existence_uniqueness}
\end{proposition}

\begin{proof}
	Note that if $u_{c}$ is a solution of (\ref{eq:K_fixed}) for $c<0$, then the function $u_{c}+\log(-c)$ will solve (\ref{eq:c_fixed}) with $k=-\frac{1}{c^2}$.
	Conversely, a solution $v_{k}$ with $k<0$ will give a solution $v_{k}+\frac{1}{2}\log (-k)$ of (\ref{eq:K_fixed}) with $c=-\frac{1}{\sqrt{-k}}$.

	So if we establish the existence and uniqueness of (\ref{eq:K_fixed}), then we automatically know (\ref{eq:c_fixed}) has a unique solution for $k<0$. 

	Let's focus on (\ref{eq:K_fixed}) first.

	The existence of (\ref{eq:K_fixed}) is just a direct results in Rupflin's work(see Remark 2.4 in \cite{rupflin2021hyperbolic}). 

	The uniqueness of $u_{c}$ for $c\ge 0$ is still a direct consequence of Proposition 2.2
	in \cite{rupflin2021hyperbolic}.

	For the case $c<0$, things get easier since we have the Maximum Principle.
	Indeed, suppose we do have two solutions $u_1,u_2$ for a fixed $c<0$, then we know $u:=u_1-u_2$ will solve
\[
		\begin{cases}
		-\Delta u=e^{2u_2}- e^{2u_1}, & \text{ in }\Sigma,\\
		\frac{\partial u}{\partial \tau}
		=ce^{u_1}-ce^{u_2}, & 
		\text{ on }\partial \Sigma.
		\end{cases}
	\]

	If $u$ attains a positive maximum at an interior point $x_0$, then $-\Delta u(x_0)\ge 0$. But on the other hand $e^{2u_2(x_0)}- e^{2u_1(x_0)}<0$ since $u_1(x_0)>u_2(x_0)$. So $u$ cannot attain a positive maximum at an interior point.
	Similarly, $u$ cannot attain a negative minimum at an interior point.

	On the boundary, we have a similar argument. If $u$ attains a positive maximum on a boundary point $x_0$, then $\frac{\partial u}{\partial \tau}(x_0)\ge 0$.
	But $ce^{u_1}-ce^{u_2}<0$ at $x_0$ on the other hand. This contradiction shows $u$ cannot attain its positive maximum on the boundary. Similarly, $u$ cannot attain a negative minimum on the boundary.

	This shows the only possible way to take $u$ is identical to zero.

	So we prove the uniqueness for the case $c<0$ for problem (\ref{eq:K_fixed}).

	In addition, we need to consider the case $k=0$ for problem (\ref{eq:c_fixed}). Up to a rescaling, we only need to prove the existence and uniqueness of a metric with zero Gaussian curvature and (negative) constant geodesic boundary curvature. This result has been done in paper \cite{osgood1988extremals}.

	Regarding the regularity part, it is a standard result in PDEs.
	For example, we can apply the regularity theorem by P. Cherrier \cite{cherrier1984problemes}.
\end{proof}

Now after getting the existence and uniqueness, we can study the properties of the solution $u_{c}, v_{k}$.
The first step might be we want to show the solutions $u_{c}, v_{k}$ continuously depend on $c$ or $k$ respectively.
\begin{proposition}
	Let $u_{c}, v_{k}$ be defined in Proposition \ref{prop:existence_uniqueness}. 

	Then both $c\rightarrow u_{c}, k\rightarrow u_{k}$ are continuous maps from $(-\infty,0]$ to $H^{1}(\Sigma,g)$. Moreover, $u_{c}\left( v_{k} \right) $ is strictly increasing pointwise as $c (k \text{ respectively})$ increasing.
	\label{prop:continous}
\end{proposition}

\begin{proof}
	We only prove the properties for $u_{c}$ since the method here is essentially the same.

	First, let's show $u_{c}$ is strictly increasing pointwise. 
	
	For any $\delta>0, c+\delta\le 0$, we consider $w = u_{c+\delta}- u_{c}$. Then $w$ will solve
	\begin{align}
		\begin{cases}
		-\Delta w = e^{2u_{c}}- e^{2u_{c+\delta}}, & 
		\text{ in }\Sigma,\\
		\frac{\partial w}{\partial \tau}=(c+\delta)e^{u_{c+\delta}}- ce^{u_{c}}, & \text{ on }\partial \Sigma.
		\end{cases}
		\label{eq:pf_cont}
	\end{align}
	
	We can apply the standard weak Maximum Principle to get $w\ge 0$ (Show that $w$ cannot have a negative minimum at both the interior points and the boundary points). So $\Delta w\ge 0$. Hence, the strong Maximum Principle implies $w$ cannot attain 0 interiorly. If $w$ attains 0 at $x_0$ for $x_0$ on the boundary, we know
\[
		\frac{\partial w}{\partial \tau}\le 0.
	\]

	But at the same time,
\[
		(c+\delta)e^{u_{c+\delta}(x_0)}- ce^{u_{c}(x_0)}= \delta e^{u_{c}(x_0)}>0
	\]
	since $u_{c+\delta}(x_0)=u_{c}(x_0)$.
	So $w$ cannot attain 0 on the boundary, either. 

	In summary, we know $w>0$. Hence $u_{c}$ is strictly increasing when $c \in (-\infty, 0]$. 
	Similarly, we can show that $v_{k}$ is strictly increasing when $k \in (-\infty,0]$.

	As a direct consequence of monotonicity, we can bound the solutions $u_{c}$ to show $c\rightarrow u_{c}$ is continuous. 

	Firstly, we have
\[
		\sup_{\Sigma}u_{c}\le \sup_{\Sigma}u_{0}=0
	\]
	since $u_{0}\equiv 0$.
	
	On the other hand, we know
\[v_{-\frac{1}{c^2}}=u_{c}+\log(-c)
	\]
	in the proof of Proposition \ref{prop:existence_uniqueness}. By the monotonicity properties of $v_{k}$, we know for any $c\le -1$, we have
\[
		u_{c}+\log(-c)= v_{-\frac{1}{c^2}}\ge v_{-1}=u_{-1}.
	\]
	
	Hence
\[
		\inf _{\Sigma}u_{c}\ge \inf _{\Sigma}u_{-1}-\log(-c)= C-\log(-c) \text{ for all }
		c\le -1.
	\]

	So in summary, we have
\[
		\|u_{c}\|_{L^\infty(\Sigma, g_0)}\le \max \{ C,C-\log(-c) \}
		\text{ for all }c \in (-\infty,0]
	\]
	for some constant $C \in \mathbb{R} $.

	Now let us use $w=u_{c+\delta}-u_{c}$ as a test function in equations (\ref{eq:pf_cont}). Here we only assume $c\le 0, c+\delta\le 0$. 
	After integration by part, we have the following energy estimate,
	\begin{align*}
		\int_{ \Sigma} 
		\left|\nabla w\right|^2dV={} & 
		\int_{ \Sigma} (e^{2u_{c}}- e^{2u_{c+\delta}})wdV+ \int_{ \partial \Sigma} \!\!
		c(e^{u_{c+\delta}}-e^{u_{c}})
		w+\delta e^{u_{c+\delta}}wds\\
		\le{} & \delta \int_{ \partial \Sigma} 
		e^{u_{c+\delta}}wds\le \delta \int_{ \Gamma} wds\\
		\le{}&
		\left|\delta\right|
		L(-1,0)\max \{ \|u_{c}\|_{L^{\infty}
	(\Sigma,g)}, \|u_{c+\delta}\|_{L^{\infty}
	(\Sigma,g)}\}.
	\end{align*}
	where we've used the monotonicity of $u_{c}$ and $u_{c}\le 0$ when $c\le 0$.

	Since $\|u_{c}\|_{L^{\infty}(\Sigma,g)}$ is bounded locally near $c$, we have
\[
		\|Dw\|^2_{L^2(\Sigma,g)}\le C \left|\delta\right|
	\]
	for $\delta$ small enough and some constant $C$ might depending on $c$.

	To finish the continuity of $c\rightarrow u_{c}$, let's show $\|w\|_{L^{\infty}(\Sigma,g)}$ can be bounded by some constant related to $\delta$.

	From the rest of this proof, we will work in the metric $g_{-1,c}$ instead of the background metric $g$.
	Then we know $w$ solve the following problem,
\[
		\begin{cases}
		-\Delta_{c}w=1- e^{2w}, & \text{ on }\Sigma, \\
		\frac{\partial w}{\partial \tau _{c}}=(c+\delta)e^{w}-c, & 
	\text{ on }\partial \Sigma.
		\end{cases}
	\]
	where $\Delta_{c}$ is the Laplacian-Beltrami operator under metric $g_{-1,c}$ and $\tau_{c}$ is the unit outer normal of $\partial \Sigma$ in $\Sigma$ under metric $g_{-1,c}$.

	Let's choose an auxiliary function $h$ which satisfies the following problem,
	\begin{align}
		\begin{cases}
		\Delta_{c}h=H_0, & 
		\text{ in }\Sigma,\\
		\frac{\partial h}{\partial \tau _{c}}=1, & \text{ on }\partial \Sigma,
		\end{cases}
		\label{eq:pf_aux_fun}
	\end{align}
	with a suitable $H_0 \in \mathbb{R} $.
	Note that problem (\ref{eq:pf_aux_fun}) has a solution if and only if $H_0A(-1,c)=L(-1,c)$. Although the solution of problem (\ref{eq:pf_aux_fun}) is not unique for such $H_0$, we just choose arbitrary one and denote it as $h$. 

	Let $ \overline{H}=\sup_{\Sigma}\left|h\right|$.
	We will show that $\left|w\right|\le 2 \left|\delta \right|	e^{C}(H_0+\overline{H})$. Actually, we only need to consider the case $\delta>0$ and it is enough to show $w\le 2 \delta e^{C}(H_0+\overline{H})$ for some $C$ large.

If not, we suppose $w>2\delta e^{C}(H_0+\overline{H})$ at some point $\overline{x}_0$. Let $\overline{w}=w-\delta e^{C}h$. Suppose $\overline{w}$ takes its maximum at point $x_0$. So $\overline{w}(x_0)\ge 
w(\overline{x}_0)-\delta e^C h(\overline{x}_0)>2\delta e^{C}H_0+\delta e^{C}\overline{H}$. Hence $w(x_0)\ge 2\delta e^{C}H_0$. 

If $x_0$ is an interior point, then we have
\[
	0\le -\Delta_{c} \overline{w}(x_0)
	=1-e^{2w(x_0)}+\delta e^{C}H_0\le 1-(1+4\delta e^{C}H_0)+\delta e^{C}H_0<0,
\]
where we've used $-e^{2x}\le -1-2x$.
A contradiction. So $x_0$ can only be on boundary $\partial \Sigma$. In this case, we have
\[
	0\le \frac{\partial \overline{w}}{\partial \tau_{c}}= ce^{w(x_0)}-c+\delta e^{w(x_0)}-\delta e^{C}\le \delta \left( e^{w(x_0)}-e^{C} \right),
\]
which is impossible if we choose $C$ large enough since the $L^\infty$ norm of $w$ is bounded if $\delta$ closed to 0.

So we've proved $\|w\|_{L^{\infty} (\Sigma,g)}\le 2\left|\delta\right| e^C(H_0+\overline{H})$. Combining previous result of estimation for $Dw$, we know $c\rightarrow u_{c}$ is indeed a continuous map from $(-\infty,0]\rightarrow H^{1}(\Sigma,g)$.

Similar methods can be applied to the case of $v_{k}$ to show $k\rightarrow v_{k}$ is a continuous map from $(-\infty,0]$ to $H^{1}(\Sigma,g)$.
We'll omit the details here. 
\end{proof}

From the above proof, we can see that the supremum of $u_{c}$ is locally bounded.
That is, $\|u_{c+\delta}\|_{L^\infty(\Sigma, g)}\le C$ for some $C$ which might be related to $c$ and all $\left|\delta\right| <1$. Since we have the regularity result for $u_{c}$ and some other related functions, we will write $\|u_{c}\|_{\infty}$ instead of $\|u_{c}\|_{L^\infty(\Sigma, g)}$ for short.

Now let's establish the following energy estimate.

\begin{proposition}
	Under the metric $g_{-1,c}$ for $c\le 0$, we suppose $w$ solves the following problem,
	\begin{align}
		\begin{cases}
		-\Delta_c w+2w=f, & \text{ in }\Sigma, \\
		\frac{\partial w}{\partial \tau_c}- c w=g, & \text{ on }\partial \Sigma \end{cases}
		\label{eq:prop_energy_est}
	\end{align}
	for some $f \in C^\infty(\Sigma), g \in C^\infty(\partial \Sigma)$.

	Suppose $\|f\|_{\infty}, \|g\|_{\infty}\le C_{1}\delta^2$ for a fixed constant $C_{1}$ and $\left|\delta \right|<1$, then we have
\[
		\|w\|_{\infty}\le C \delta^2
	\]
	where $C$ does not rely on $\delta$ for some constant $C$. $C$ might be related to $C_{1}$.

	Moreover, we have
\[
		\|w\|_{H^{1}(\Sigma, g_{-1,c})}\le C\left( \|f\|
	_{L^2(\Sigma,g_{-1,c})}+ \|g\|_{L^2(\partial \Sigma,g_{-1,c})}\right) 
+C\delta^2.
	\]

	The $C$ here is still not related to $\delta$.
	\label{prop_H1_estimate}
\end{proposition}

\begin{proof}
	Again we choose an auxiliary function $h$ satisfying (\ref{eq:pf_aux_fun}) and we consider the function $\overline{w}= l h+A\pm w$ with $l,A$ decided later on.

	So $\overline{w}$ will solve the following problem
	\begin{align}
		\begin{cases}
		-\Delta_c \overline{w}+2\overline{w}= 2lh+2A-lH_0\pm f, & \text{ in }
		\Sigma,\\
		\frac{\partial \overline{w}}{\partial \tau_c}-c\overline{w}=l-clh- cA\pm g, & \text{ on }\partial \Sigma.
		\end{cases}
		\label{eq:pf_aux_2_fun}
	\end{align}
	
	Again we let $\overline{H}= \sup_{\Sigma}\left|h\right|$. 

	If $c< -\frac{1}{2\overline{H}}$, we can choose $l=0, A = \max \{ \|f\|_\infty , 2\|g\|_\infty \overline{H}\}$.
	Then we have
\[
		-\Delta_c \overline{w}+2\overline{w}\ge 0\quad \text{ and }\quad \frac{\partial \overline{w}}{\partial \tau_c}
		-c \overline{w}\ge 0.
	\]

	So $\overline{w}$ cannot have a negative minimum in $\Sigma$ by the standard Maximum Principle.

	If $-\frac{1}{2\overline{H}}\le c\le 0$, then we choose $l=2\|g\|_\infty ,A = l(2\overline{H}+H_0)+ \|f\|_\infty $. We will have
	\begin{align*}
		-\Delta_c \overline{w}+2\overline{w}\ge{} & 
		-2l\overline{H}+l(2\overline{H}+H_0)+ \|f\|_\infty -l H_0\pm f\ge 0,\\
		\frac{\partial \overline{w}}{\partial \tau_c}-c\overline{w}\ge{} & l-\frac{1}{2\overline{H}}l\overline{H} \pm g\ge 0.
	\end{align*}

	So we can also get $\overline{w}$ cannot have a negative minimum in $\Sigma$. To be more precise here, we know the condition $-\Delta_c\overline{w}+2\overline{w}\ge 0$ in $\Sigma$ will imply $\overline{w}$ cannot attain a negative minimum in the interior by the Weak Maximal Principle. If $\overline{w}$ attains a negative minimum at a boundary point $x_0$, then we know
	\[
		\overline{w}(x)>\overline{w}(x_0)\quad \text{ for any }x
		\text{ in the interior of }\Sigma.
	\]

	Since $\overline{w}(x_0)<0$, the Hopf's Lemma can be applied here to get $\frac{\partial \overline{w}(x_0)}{\partial \tau_c}<0$.
	This will imply
	\[
		\frac{\partial \overline{w}(x_0)}{\partial \tau_c}-
		c\overline{w}(x_0)\le \frac{\partial \overline{w}(x_0)}{\partial \tau_c}<0,
	\]
	which leads to a contradiction.

	So the above results give us
	\[
		\left|w\right|\le l\overline{H}+A\le C\delta^2
	\]
	by the choice of $l,A$ and the conditions $\|f\|_\infty ,\|g\|_\infty \le C_{1}\delta^2$.

	At last, we use $w$ as a test function in problem (\ref{eq:prop_energy_est}), then we have
	\begin{align*}
		\int_{ \Sigma} 
		\left|\nabla w\right|^2+2w^2dV={} & 
		\int_{ \Sigma} fwdV+\int_{ \partial \Sigma} 
		gw+cw^2ds\\
		\le{} & \frac{1}{2}\int_{ \Sigma} 
		f^2dV+\frac{1}{2}\int_{ \Sigma} w^2dV
		+\frac{1}{2}\int_{ \partial \Sigma} 
		g^2ds+\frac{1}{2}\int_{ \partial \Sigma} 
		w^2\\
		\le{}& \frac{1}{2}\|f\|^2_{L^2(\Sigma,g _{-1,c})}+ \frac{1}{2}\|g\|^2_{L^2(\partial \Sigma,g _{-1,c})}+\frac{1}{2}\|w\|^2_{L^2(\Sigma,g _{-1,c})}+C\delta^4
	\end{align*}
	where we've used $\|w\|_{\infty}\le C\delta^2$. So
\[
		\|w\|_{H^{1}(\Sigma, g_{-1,c})}\le C\left( \|f\|
	_{L^2(\Sigma,g_{-1,c})}+ \|g\|_{L^2(\partial \Sigma,g_{-1,c})}\right) 
+C\delta^2.
	\]
\end{proof}

With this energy estimate, we can show that the map $c\rightarrow u_{c}$ is $C^{1}$.

\begin{proposition}
	Let $u_{c}$ be defined in Proposition \ref{prop:existence_uniqueness}. Then $c\rightarrow u_{c}$ is a $C^{1}$ map from $(-\infty,0]$ to $H^{1}(\Sigma,g)$.
	\label{prop_C1_uc}
\end{proposition}
\begin{proof}
	We still work under the metric $g_{-1,c}$.
	Let's consider the following problem (linearised problem of (\ref{eq:K_fixed})),
	\begin{align}
		\begin{cases}
		-\Delta_c w=-2w, & \text{ in }\Sigma, \\
		\frac{\partial w}{\partial \tau_c}=cw+1, & 
		\text{ on }\partial \Sigma.
		\end{cases}
		\label{eq:pf_linearised_problem}
	\end{align}
	
	We show that it has a unique solution $w$.
	Indeed, uniqueness is a direct consequence of Proposition \ref{prop_H1_estimate}. 
	For the existence, we use the following variational integral
\[
		I(w):=\int_{ \Sigma} 
		\left( \left|\nabla w\right|^2+2w^2 \right) 
		dV-\int_{ \partial \Sigma} \left( c
		w^2+w\right) ds.
	\]

	It is coercive by the trace theorem and $c\le 0$. Then $I(w)$ will have a minimizer $w$ which of course solves (\ref{eq:pf_linearised_problem}). Standard regularity results imply $w$ is smooth.
	
	Now we consider the function $\overline{w}:=u_{c+\delta}-u_{c}- \delta w$ for any $c,c+\delta\le 0$ and $\left|\delta\right|<1$.

	$\overline{w}$ will solve the following problem
	\begin{align*}
		\begin{cases}
			-\Delta_c \overline{w}-2\overline{w}
			^2=1-e^{2v}+2v, & 
			\text{ in }\Sigma,\\
		\frac{\partial \overline{w}}{\partial \tau_c}= c(e^{v}-1-v)+\delta(e^{v}-1), & 
			\text{ on }\partial \Sigma, \end{cases}
	\end{align*}
	where $v=u_{c+\delta}-u_{c}$. This means $\overline{w}$ is a solution of (\ref{eq:prop_energy_est}) with $f = 1-e^{2v}+2v, g=c(e^{v}-1-v)+\delta (e^{v}-1)$.

	Note that we already have $\|v\|_\infty \le C\left|\delta\right|$. So for $f,g$, using $ \left|e^x-1\right|\le Cx, \left|e^{x}-1-x\right|\le Cx^2$ for $x$ closed to 0, we have
\[
		\|f\|_\infty \le C\delta^2,\quad \|g\|_\infty \le \left|c\right|
		C\delta^2+\delta(C \delta)\le C \delta^2.
	\]

	After applying Proposition \ref{prop_H1_estimate}, we know
\[
		\|u_{c+\delta}-u_{c}-\delta w\|_{H^{1}(\Sigma,g_{-1,c})}\le C \delta^2
	\]
	
	This is true under metric $g$, too since $\|u_{c}\|_\infty <\infty$.
	So $c\rightarrow u_{c}$ is a $C^{1}$ map from $(-\infty,0]$ to $H^{1}(\Sigma,g)$.
\end{proof}

\begin{remark}
	The above proof indeed shows $w$ is the differential of the map $c\rightarrow u_{c}$ at the point $c$.
\end{remark}

Similarly, we can apply these methods to the dual case.
\begin{proposition}
	Under the metric $g_{k,-1}$, we suppose $w$ solves the following problem,
	\begin{align}
		\begin{cases}
		-\Delta w-2k w=f, & \text{ in }\Sigma, \\
		\frac{\partial w}{\partial \tau}+w=g, & 
		\text{ on }\partial \Sigma \end{cases}
	\end{align}
	for some $f \in C^{\infty}(\Sigma), g \in C^{\infty}(\partial\Sigma)$.

	Suppose $\|f\|_\infty ,\|g\|_\infty \le C_1\delta^2$ for some fixed constant $C_1$ and $\left|\delta\right|<1$, then we have
	\[
		\|w\|_{\infty}\le C \delta^2
	\]
	where $C$ does not rely on $\delta$ for some constant $C$.
	Moreover, we have
\[
		\|w\|_{H^{1}(\Sigma,g_{k,-1})}\le C\left( \|f\|_{L_2(\Sigma,g_{k,-1})}
		+\|g\|_{L_2(\partial\Sigma,g_{k,-1})}\right) 
		+C\delta^2.
	\]

	Using these results, we can show that, if $v_{k}$ is defined in Proposition \ref{prop:existence_uniqueness}, then $k\rightarrow v_{k}$ is a $C^{1}$ map from $(-\infty,0]$ to $H^{1}(\Sigma,g)$.
\end{proposition}

Now we can finish the proof of our main theorems. 
\begin{proof}
	[Proof of Theorem \ref{thm:fix_K} and Theorem \ref{thm:fix_c}]
	The existence and uniqueness of metric $g_{-1,c} (g_{k,-1})$ is equivalent to the existence and uniqueness of $u_{c}
	(v_{k})$, which has been proved in Proposition \ref{prop:existence_uniqueness}.
	Those metrics can be directly chosen by
\[
		g_{-1,c}=e^{2u_{c}}g,\quad g_{k,-1}=e^{2v_{k}}g.
	\]

	Note that
\[
		L(-1,c)=\int_{ \partial \Sigma} 
		e^{u_{c}}ds,\quad A(-1,c)=\int_{ \Sigma} e^{2u_{c}}
		dV. \]

	So they are all continuous functions defined on $(-\infty,1)$ by the continuity of $u_{c}$.
	We know both $L(-1,c),A(-1,c)$ is strictly increasing when $c \in (-\infty,0]$ by the monotonicity properties of $u_{c}$. The Gauss-Bonnet formula tells us
	\begin{align}
		\hat{L}(-1,c)-A(-1,c)=2\pi \chi(\Sigma)
		\label{eq:pf_gauss_bonnet}
	\end{align}
	
	So $\hat{L}(-1,c)$ is strictly increasing when $c \in (-\infty,0]$. 
	For the case $c \in [0,1)$, we need use the result in \cite{rupflin2021hyperbolic}. 
	Indeed, one of the main results of Rupflin's paper implies $c\rightarrow c L(-1,c)$ is a $C^{1}$ strictly monotonicity increasing functions that have a limit
	\[
		\lim_{c\rightarrow 1^-} c
		L(-1,c)=+\infty.
	\]
	See Lemma 3.1 in \cite{rupflin2021hyperbolic} for details. Besides, the monotonicity of $L(-1,c)$ is still a corollary of Rupflin's result. Note that under metric $g_{-1,c}$, we have
	\[
		\frac{d}{dc}L(-1,c)= \int_{ \partial \Sigma} wds 
	\]
	where $w$ is still the function defined by problem (\ref{eq:pf_linearised_problem}). 
	Following the steps in the proof of Lemma 3.1 in \cite{rupflin2021hyperbolic}, we can apply the trace estimate to give
	\begin{align*}
		\int_{ \Sigma} 
		\left|\nabla w\right|^2+2w^2 dV={} & 
		\int_{ \partial \Sigma} cw^2+wds\\
		\le{} & (1-\varepsilon ) \int_{ \Sigma} 
		\left|\nabla w\right|^2+2w^2dV+ \int_{ \partial \Sigma} wds \end{align*}

	This will imply $\int_{ \partial \Sigma} wds>0$.
	Hence $L(-1,c)$ should be strictly increasing in $[0,1)$ and $L(-1,1)=+\infty$ since $\hat{L}(-1,1)=+\infty$.
	Again by the Gauss-Bonnet formula (\ref{eq:pf_gauss_bonnet}), we know $A(-1,c)$ is strictly increasing when $c \in [0,1)$ and $A(-1,1)=+\infty$.

	This will finish the proof of Theorem \ref{thm:fix_K}.

	For the rest of the proof, we note the continuity and monotonicity of $L(k,-1)$, $A(k,-1)$ is a direct consequence of monotonicity of $v_{k}$, which has been proved in Proposition \ref{prop:continous}.
	So the monotonicity of $\hat{A}(k,-1)$ coming from the Gauss-Bonnet formula
\[
		\hat{A}(k,-1)-L(k,-1)= k A(k,-1)-L(k,-1)=2\pi \chi(\Sigma).
	\]
	
	Take $k=0$ in the above Gauss-Bonnet formula, we have
\[
		\hat{A}(0,0)=0\quad \text{ and }
		\quad L(0,-1)=-2\pi\chi(\Sigma).
	\]

	So we finish the proof of theorem \ref{thm:fix_c}.
\end{proof}
\begin{remark}
	We could show the function $L(-1,c), A(-1,c), L(k,-1),A(k,-1)$ are all at least $C^1$ for $c \in (-\infty,1), k \in (-\infty,0]$ since we know $c\rightarrow u_c$ and $k\rightarrow v_k$ are at least $C^{1}$. Since we do not use this result here, we will leave this result to the readers.
\end{remark}

\begin{proof}
	[Proof of Theorem \ref{thm:uniform_surface}]
	Existence and uniqueness of $g_{k,c}$ is a combination of Theorem \ref{thm:fix_K} and Theorem \ref{thm:fix_c} using the scaling properties
\[
		\lambda^2g_{\lambda^2k,\lambda c}= g_{k,c} \text{ for any }\lambda>0
	\]

	For the function $L, A, \hat{L}$ and $A$ defined on $\mathcal{D}$, we have
\begin{align}
\arraycolsep=1.4pt
		\begin{array}{rllrrll}
			L\left( \frac{k}{\lambda^2}, \frac{c}{\lambda^2}\right) &=& \lambda L(k,c)& \quad & 
			\hat{L}\left( \frac{k}{\lambda^2}, \frac{c}{\lambda^2}\right) &=& \hat{L}(k,c)\\
			A\left( \frac{k}{\lambda^2}, \frac{c}{\lambda^2}\right) &=& \lambda^2 A(k,c)& \quad & 
			\hat{A}\left( \frac{k}{\lambda^2}, \frac{c}{\lambda^2}\right) &=& \hat{A}(k,c)
		\end{array}
		\label{eq:pf_scal}
	\end{align}
	for any $\lambda>0$. So all of these functions are continuous on $\mathcal{D}$. 

	The monotonicity of these four functions are direct consequences of Proposition \ref{prop:continous} and equations (\ref{eq:pf_scal}). 
	The limited behavior of these four functions is still easy to see.
	For completeness, let's analyze them one by one.

	\noindent\textbf{Monotonicity of $L$.}

	First, let's show $L$ is strictly increasing with respect to $c$. 

	For $k<0$, by equations (\ref{eq:pf_scal}), we have
	\[
		L(k,c)=\frac{1}{\sqrt{-k}}
		L\left(-1,\frac{c}{\sqrt{-k}}\right)
	\]
	which is strictly increasing when $c$ increases.

	For $k=0$, we have $L(0, c)=\frac{2\pi \chi(\Sigma)}{c}$, which is strictly increasing when $c<0$ by the Gauss-Bonnet formula.

	Second, let's show $L(k,c)$ is strictly increasing when $c$ is fixed. We consider the following three cases.
	\begin{itemize}
		\item If $c<0$, we have $L(k,c)=-\frac{1}{c}L\left(\frac{k}{c^2},-1\right)$. So for the fixed $c$, the monotonicity of $L(k,c)$ comes from the monotonicity of $L(k, -1)$, which is strictly increasing.
		\item If $c=0$, we have $L(k,0)=-\frac{1}{k}L(-1,0)$, which is strictly increasing for $k \in (-\infty,0)$.
		\item If $c>0$, we note $cL(k,c)=\frac{c}{\sqrt{-k}}L\left( -1,\frac{c}{\sqrt{-k}} \right) $. By the monotonicity of $\hat{L}(-1,c)$ when $c \in (0,1)$, we know $cL(k,c)$ is strictly increasing when $k \in (-\infty,-c^2)$. So $L(k,c)$ is strictly increasing with respect to $k$.
	\end{itemize}

	So we know $L$ will be strictly increasing along both positive directions of $c$ and $k$.
	
	\noindent\textbf{Limit behavior of $L$.}

	For $(k,c)\rightarrow (-\lambda^2,\lambda)$ with $\lambda>0$, we have
\[
		L(k,c)=\frac{1}{\sqrt{
			-k	
		}}L\left( -1, \frac{c}{\sqrt{-k}}\right)
\rightarrow +\infty \]

	For fixed $k\le 0$, we have
\[
		L(k,-\lambda)= \frac{1}{\lambda}
		L\left( \frac{k}{\lambda^2},
		-1\right) \rightarrow 0
		\text{ as }\lambda\rightarrow 
		+\infty.
	\]

	Similarly, for fixed $c \in \mathbb{R} $, we have
\[
		L(-\lambda^2,c)= \frac{1}{\lambda}
		L\left( -1,\frac{c}
		{\lambda}\right) \rightarrow 0 \text{ as }\lambda\rightarrow 
		+\infty.
	\]
	
	For the values of $L(0,c)$, we just need to apply the Gauss-Bonnet formula to see
\[
		c L(0,c)+0\times A(0,c)= 2\pi \chi(\Sigma)
	\]

	\noindent\textbf{Monotonicity of $\hat{L}$.}

	First, when fixing $k$ with $k<0$, by (\ref{eq:pf_scal}) we have
	\[
		\hat{L}(k,c)=\hat{L}\left(-1,\frac{c}{\sqrt{-k}}\right)
	\]
	which is strictly increasing for $c \in (-\infty, \sqrt{-k})$ by Theorem \ref{thm:fix_K}. Note that when $k=0$, we have
	\[
		\hat{L}(0,c)=cL(0,c)=2\pi \chi(\Sigma),
	\]
	which is a constant function with respect to $c \in (-\infty,0)$.

	Second, when fixing $c$, we have following three cases.
	\begin{itemize}
		\item If $c<0$, we have $\hat{L}(k,c)=c L(k,c)$, which is strictly decreasing by the monotonicity property of $L$. 
		\item If $c=0$, we have $\hat{L}(k,0)=0$ by definition of $\hat{L}$.
		\item If $c>0$, we have $\hat{L}(k,c)=c L(k,c)$, which is strictly increasing by the monotonicity property of $L$.
	\end{itemize}

	\noindent\textbf{Limit behavior of $\hat{L}$.}

	\begin{itemize}
		\item $\hat{L}(-\lambda^2,\lambda)=\hat{L}(-1,1)=+\infty$ by Theorem \ref{thm:fix_K}. 
		\item $\hat{L}(k,-\lambda)=\hat{L}(\frac{k}{\lambda^2},-1)=-L(\frac{k}{\lambda^2},-1)\rightarrow 2\pi\chi(\Sigma)$ as $\lambda\rightarrow +\infty$ by the limit behavior of $L$.
		\item $\hat{L}(0,c)=2\pi\chi(\Sigma)$ is followed by the Gauss-Bonnet formula.
		\item $\hat{L}(-\lambda^2,c)=cL(-\lambda^2,c)\rightarrow 0$ as $\lambda\rightarrow +\infty$ by the limit behavior of $L$.

	\end{itemize}

	\noindent\textbf{Monotonicity of $A$.}

	First, let's fix $k\le 0$. If $k<0$, then $kA(k,c)=2\pi \chi(\Sigma)-cL(k,c)$ by the Gauss-Bonnet formula, which is strictly decreasing by the monotonicity properties of $\hat{L}$. So $A(k,c)$ itself is strictly increasing. If $k=0$, then $A(0,c)=\frac{1}{c^2}A(0,-1)$ by equations (\ref{eq:pf_scal}), which is strictly increasing when $c<0$.

	When fixing $c \in \mathbb{R} $, we have following three cases.
	\begin{itemize}
		\item If $c<0$, we have $A(k,c)=\frac{1}{c^2}A(\frac{k}{c^2},-1)$, which is strictly increasing by Theorem \ref{thm:fix_c}.
		\item If $c=0$, we have $A(k,0)=\frac{2\pi\chi(\Sigma)}{k}$ by the Gauss-Bonnet formula, which is strictly increasing since $2\pi\chi(\Sigma)<0$.
		\item If $c>0$, we have $A(k,c)=\frac{1}{-k}A(-1,\frac{c}{\sqrt{-k}})$. Note that the function $\lambda^2A(-1,\lambda)$ is a strictly increasing function when $\lambda>0$ since $A(-1,\lambda)$ is strictly increasing. So $c^2A(k,c)$ is a strictly increasing function when $k \in (-\infty,-c^2)$. So $A(k,c)$ is strictly increasing when fixing $c>0$.
	\end{itemize}

	\noindent\textbf{Limit behavior of $A$.}
	\begin{itemize}
		\item By the Gauss-Bonnet formula, for $\lambda>0$, we have $-\lambda^2A(-\lambda^2,\lambda)=2\pi \chi(\Sigma)-\lambda L(-\lambda^2,\lambda)=-\infty$. So $A(-\lambda^2,\lambda)=+\infty$.
		\item $A(k,-\lambda)=\frac{1}{\lambda^2}A(\frac{k}{\lambda^2},-1)\rightarrow 0$ as $\lambda\rightarrow +\infty$.
		\item $A(0,c)=\frac{A(0,-1)}{c^2}$ by equations (\ref{eq:pf_scal}) for $c<0$.
		\item $A(-\lambda^2,c)=\frac{1}{\lambda^2}A(-1,\frac{c}{\lambda})\rightarrow 0$ as $\lambda\rightarrow +\infty$.
	\end{itemize}

	For the properties $\hat{A}$, we use the Gauss-Bonnet formula to get $\hat{A}(k,c)+\hat{L}(k,c)=2\pi\chi(\Sigma)$. So all of the properties of $\hat{A}$ come from the properties of $\hat{L}$ directly.
\end{proof}

\begin{remark}
	The above theorems and their proofs are still valid if $\partial \Sigma$ has several components. 
	But we will expect more results than that, inspired by the work of Rupflin \cite{rupflin2021hyperbolic}.

	More precisely, we suppose $\Sigma$ has boundary $\partial \Sigma= \cup _{i=1}^n \Gamma_i$ with $n\ge 2$. We have the following questions.
	\begin{itemize}
		\item For what kind of $c=(c_1,\cdots ,c_n)\in \mathbb{R}^n $ and $k\le 0$ such that there is a metric $g$ with interior Gaussian curvature $k$ and geodesic boundary curvature $c_i$ on the boundary $\Gamma_i$. 
			In general, even for hyperbolic metric, i.e. $k=-1$, $c_i$ might exceed 1 if some of other $c_j$ small enough. This is quite difference with the theorem in \cite{rupflin2021hyperbolic}.
		\item How about the uniqueness of metric if it does exist for a given $c \in \mathbb{R}^n $ and $k\le 0$? In general, uniqueness might not hold even in hyperbolic metrics if some of $c_i$ exceed 1.
			So this question indeed asks how many solutions it will have and how to find these solutions. 
		\item The last question we are concerned about is how the boundary length and area are related to the choice of $c \in \mathbb{R}^n $ and $k\le 0$.
	\end{itemize}

	These questions will be related to the weak uniformization of the triple junction surface that $\Gamma$ has more than one component.
	
	Moreover, we can ask the same questions for the case $k>0$. Even for the case $\partial \Sigma$ having only one component, we do not know the existence and uniqueness results. 
	Some examples show uniqueness will fail when $k>0$.
	\label{rmk:multiple_boundary}
\end{remark}

\section{Weak uniformization of triple junction surfaces}%
\label{sec:weak_uniformization_of_triple_junction_surface}

In this section, we will prove the weak uniformization for the triple junction surface.
In particular, we will prove the following theorem.

\begin{theorem}
	[Weak uniformization for $\chi(M)\le 0$]
	Suppose $M=(\Sigma_1,\Sigma_2,\Sigma_3;\Gamma)$ is a triple junction surface with metric $g=(g_1,g_2,g_3)$ on it and $\chi(M)\le 0$ such that $\Gamma$ has only one component. Then we can find a hyperbolic metric $\overline{g}=(\overline{g}_1, \overline{g}_2, \overline{g}_3)$ and new diffeomorphisms $\overline{\varphi}_i: \Gamma\rightarrow \partial \Sigma_i$ such that the weak uniformization holds for $k=-1$.
	
	Moreover, $\overline{g}_i$ is unique if $\chi(\Sigma_i)\le 0$.
	\label{thm_weak_uniformization_for_chi_m_le_0_}
\end{theorem}

Recall that for triple junction surface $M$ with $\Gamma$ having only one component, we know $ \chi(M)= 3-\sum_{i =1}^{3}\text{genus}(\Sigma_i) $, which can only be an odd number. So $\chi(M)\le 0$ is equivalent to say $\chi(M)<0$. Hence we can require $k=-1$ in the weak uniformization.

The proof of this theorem is a direct application of Theorem \ref{thm:uniform_surface}. 
\begin{proof}
	[Proof of Theorem \ref{thm_weak_uniformization_for_chi_m_le_0_}]
	For any $\Sigma_i$, if $\chi(\Sigma_i)\le 0$ (indeed $\chi(\Sigma)<0$), we can apply Theorem \ref{thm:uniform_surface} to get $L_i(-1,c)$ is a strictly increasing function from $(-\infty,1)$ to $(0,+\infty)$ where $L_i$ is the length of $\Gamma$ after uniformization.
	So the inverse function of $L_i(-1,c)$ exist. That means, we can find a continuous function $c_i(l): (0,+\infty)\rightarrow (-\infty,1)$ which is strictly increasing such that $L_i(-1,c_i(l))=l$. The geometric meaning of this function $c_i(l)$ is, given any $l \in (0,+\infty)$, we can find a unique hyperbolic metric $\overline{g}_i$ conformal to $g_i$ such that the boundary length of $\Gamma$ under metric $\overline{g}_i$ is $l$. 

	Now we define the function $\hat{c}_i(l)= lc_i(l)$ for $l \in (0,+\infty)$. It is continuous and strictly increasing by the properties of function $\hat{L}$ in Theorem \ref{thm:uniform_surface}. Moreover, we have the following limit behavior of $\hat{c}_i$ as
\[
		\lim_{l\rightarrow 0} \hat{c}_i(l)= 2\pi \chi(\Sigma),\quad \lim_{l\rightarrow +\infty} \hat{c}_i(l)=
		+\infty.
	\]

	If $\chi(\Sigma_i)>0$, then $\Sigma_i$ is diffeomorphic to a disk. Since any metric on disk will conformal to each other, we can consider the metric
\[
		\frac{4\rho \left|dz\right|^2}{\left( 1-\rho \left|z\right|^2\right) ^2}
		\text{ on unit disk }\mathbb{D}
	\]
	with $0<\rho<1$.
	This is a hyperbolic metric on $\mathbb{D}$ with boundary length $L_{\mathbb{D}}(\rho)= \frac{4\pi \sqrt{\rho}}{1-\rho}$ and geodesic boundary curvature $\kappa_{\mathbb{D}}(\rho)=\frac{1+\rho}{2\sqrt{\rho}}$.
	Hence, we can see that the function
\[
		c _{\mathbb{D}}(l):= \kappa_{\mathbb{D}}
		\circ L^{-1}_{\mathbb{D}}(l),\quad l \in (0,+\infty)
	\]
	is a strictly decreasing continuous function onto the set $(1,+\infty)$ after a simple calculation. The geometric meaning of $c _{\mathbb{D}}(l)$ here is, given $l \in (0,+\infty)$, we can always find a hyperbolic metric $\overline{g}$ on $\mathbb{D}$ such that $\overline{g}$ has constant geodesic curvature $c _{\mathbb{D}}(l)$ on $\partial \mathbb{D}$ with boundary length equals to $l$. Moreover, such of $\overline{g}$ is unique upto a M\"obius transformation on disk by the conformal structure on disk. So $c _{\mathbb{D}}(l)$ is completely determined by $l$.

	We can still define $\hat{c}_{\mathbb{D}}(l)$ by $\hat{c}_{\mathbb{D}}=lc _{\mathbb{D}}(l)$.
	If we write it out, we will find
\[
		\hat{c}_{\mathbb{D}}(l)= \left( \frac{2\pi(1+\rho)}{1-\rho} \right) 
		\circ L^{-1}_{\mathbb{D}}(l).
	\]

	So $\hat{c}_{\mathbb{D}}$ is a strictly increasing continuous function which maps $(0,+\infty)$ onto the set $(2\pi,+\infty)$.
	Note that $2\pi\chi(\mathbb{D})=2\pi$ here.

	So in summary, for any $\Sigma_i$, we can define the continuous function $\hat{c}_i(l)$ by choose $\hat{c}_i(l)=\hat{c}_{\mathbb{D}}(l)$ if $\Sigma_i$ is diffeomorphic to a disk.
	
	Hence the function
	\begin{align*}
		\hat{c}(l):(0,+\infty)
		\rightarrow {} & (2\pi\chi(M),+\infty) \\
		l\rightarrow {} & \sum_{i =1}^{3}\hat{c}_i(l)
	\end{align*}
	is a strictly increasing function. So there is a unique $l_0$ such that $\hat{c}(l_0)=0$.

	With such $l_0$, we know
\[
		l_0\sum_{i =1}^{3}c_i(l_0)=0
	\]
	
	According to the geometric meaning of $c_i$, we know we can find hyperbolic metrics $\overline{g}=(\overline{g}_1, \overline{g}_2,\overline{g}_3)$ such that $\overline{g}_i$ has boundary length $l_0$ and geodesic curvature of boundary $c_i(l_0)$ such that $\sum_{i =1}^{3}c_i(l_0)$. Since all boundaries have the same length, we can construct new diffeomorphisms $\overline{\varphi}_i: \Gamma\rightarrow \partial \Sigma_i$ to make sure they are all isometric by imposing a suitable metric on $\Gamma$.
	Clearly, this metric $\overline{g}$ meets all the conditions in the definition of weak uniformization of triple junction surfaces.

	After establishing the existence, the uniqueness is easy to see. Note that if we have another metric $\hat{g}= (\hat{g}_1,\hat{g}_2,\hat{g}_3)$ such that it is a weakly uniformization of $M$, then the first thing we note is the length of $\Gamma$ should be $l_0$.
	This is because the geodesic curvature of boundary $\kappa_i$ is still given by $c_i(l)$ by the uniqueness of uniformization of surface with boundary (Upto a M\"obius transformation for $\chi(\Sigma_i)>0$, but it will keep the geodesic curvature and the length of the boundary).
	But we know the only length $l$ to make $\hat{c}(l)=0$ is $l_0$. The same length on the boundary on the other hand will give the uniqueness of hyperbolic metric on $\Sigma_i$ if $\chi(\Sigma_i)\le 0$.
\end{proof}

\begin{remark}
	If we only consider the weak uniformization of triple junction surface $M$ with only one component of triple junction such that all of $\chi(\Sigma_i)\le 0$ (Indeed $\chi(\Sigma_i)<0$), then there is a very quick way to do it.

	For each $c \in (-\infty, 1)$, we can consider the energy
\[
		I_i(u,c)=\int_{ \Sigma_i}
		\frac{1}{2}\left|\nabla u\right|^2+ \frac{1}{2}e^{2u}-K_{g_i}u dV+ \int_{ \partial \Sigma_i} \kappa _{g_i}u-e^{u}ds \]

	Now we can consider the min-max point of the following.
	\begin{align}
		\max _{\sum_{i =1}^{3}c_i=0, c_i<1
		}\sum_{i =1}^{3}
		\inf _{u_i \in H^{1}(\Sigma_i,g_i)}I_i(u_i,c_i).
		\label{eq:min_max}
	\end{align}

	Using the results in \cite{rupflin2021hyperbolic}, we can show all of the function
\[
		c_i \rightarrow \inf_{u_i \in H^{1}(\Sigma _i,g)}I_i(u_i,c_i)
	\]
	are continuous with limit $+\infty$ when $c_i\rightarrow 1$.

	We can let $u_c$ be the minimizer of $I_i(u,c)$ for $u \in H^{1}(\Sigma,g)$. So $c\rightarrow u_c$ is a $C^{1}$ map to $H^{1}(\Sigma_i,u)$.
	So the quantity (\ref{eq:min_max}) will take its maximum value in the interior of set $\{ (c_1,c_2,c_3)\in \mathbb{R}^3 , c_1+c_2+c_3=0, c_i<1 \text{ for all }i\}$.
	Such of $(c_1,c_2,c_3)$ will exactly give same boundary length for metric determined by minimizer of $I_i(u_i,c_i)$.

	But if one of $\chi(\Sigma_i)>0$, then the min-max argument cannot be applied anymore.
	So we need to establish more precise relations of boundary length and geodesic curvature of boundary after uniformization, especially when geodesic curvature of boundary approaches to $-\infty$. This is also the reason why we spend so much space talking about the uniformization of surfaces with boundary.
\end{remark}

\subsection*{Acknowledgements}
	I would like to thank my advisor Prof. Martin Li for his helpful discussions and encouragement.

The author is partially supported by a research grant from the Research Grants Council of the Hong Kong Special Administrative Region, China [Project No.: CUHK 14304120] and CUHK Direct Grant [Project Code: 4053401]. 

\bibliographystyle{plain}
\bibliography{references}

\end{document}